\documentclass[12pt,a4paper]{article}
\usepackage{fontenc}
\usepackage[utf8]{inputenc}
\usepackage{amssymb}
\usepackage{amsmath}
\usepackage{amsthm}
\usepackage{indentfirst}
\usepackage{ot1patch}
\usepackage{psfrag}
\usepackage{rotate}
\usepackage{hyperref}
\usepackage{pdflscape}
\usepackage{multirow}
\theoremstyle{plain}
\newtheorem{tw}{Theorem}[section]
\newtheorem{pr}[tw]{Proposition}
\newtheorem{lm}[tw]{Lemma}

\theoremstyle{definition}

\theoremstyle{remark}
\newtheorem{ex}[tw]{Example}
\newtheorem{con}[tw]{Conjecture}

\DeclareMathOperator{\Irr}{Irr}
\DeclareMathOperator{\Sqf}{Sqf}

\DeclareMathOperator{\Gpr}{Gpr}

\DeclareMathOperator{\GCD}{GCD}
\DeclareMathOperator{\LCM}{LCM}
\newcommand{\vvv}{{\it v}}

\usepackage{tikz}

\newcommand*\circled[1]{\tikz[baseline=(char.base)]{
            \node[shape=circle,draw,inner sep=1.5pt] (char) {{\rm #1}};}}

\DeclareMathOperator{\rpr}{\rm rpr}

\author{\L ukasz Matysiak\\
	Kazimierz Wielki University\\
	Bydgoszcz, Poland \\
	lukmat@ukw.edu.pl}
\title{On square-free and radical factorizations and existence of some divisors}

\begin{document}

\maketitle

\begin{abstract}
	We discuss various square-free and radical factorizations and existence of some divisors in monoids in the context of: atomicity,
	ascending chain condition for principal ideals, a pre-Schreier property, a greatest common
	divisor property and a greatest common divisor for sets property. 
\end{abstract}

\begin{table}[b]\footnotesize\hrule\vspace{1mm}
\end{table}

\section{Introduction}

Let $\mathbb{N}=\{1, 2, \dots \}$ and $\mathbb{N}_0=\{0, 1, 2, \dots \}$.

\medskip

Throughout this paper by a monoid we mean a commutative cancellative monoid.

\medskip

Let $H$ be a monoid. We denote by $H^{\ast}$ the group of all invertible elements of $H$. 

\medskip

If $a$, $b\in H$ are relatively primes in $H$, i.e. do not have a common invertible divisor of $H$, then we write $a\rpr b$. Therefore, if $M$ be a submonoid of $H$ and elements $a$, $b\in M$ are relatively primes in $M$, then we write $a\rpr_M b$.

\medskip

If $a$, $b\in H$ satysfying the condition $a=ub$, where $u\in H^{\ast}$, then we write $a\sim b$. 

\medskip

The set of all irreducible elements (atoms) of $H$ will be denoted by $\Irr H$. Recall that an element $a\in H$ is called square-free if it cannot be presented in the form $a=b^2c$, where $b$, $c\in H$ and $b\notin H^{\ast}$. The set of all square-free elements of $H$ we will denote by $\Sqf H$.

\medskip

In \cite {3} Theorem 5.1, with co-authors P. Jędrzejewicz, M. Marciniak and J. Zieliński, we presented a full characterization of submonoids $M$ of the factorial monoid $H$ satisfying the condition 
\begin{itemize}
	\item[(1) ] $\Sqf M\subset\Sqf H$
\end{itemize}
assuming $M^{\ast}=H^{\ast}$.

\medskip

The equivalence of (1) and 
\begin{itemize}
	\item[(2) ] for every $a\in H$, $b\in\Sqf H$, if $a^2b\in M$, then $a$, $b\in M$
\end{itemize}
in \cite{JZanalogs} has been extended to the equivalence of 8 conditions.
Two of these conditions represent a closure with respect to the 1s and 3s factorization (See section \ref{R3.1}), while the closure with respect to 3s was obtained at an earlier stage of the research and published in \cite{factprop}.

\medskip

In addition, we received a full description of such submonoids (of factorial monoid) satisfying the condition (1). They are (with an accuracy to the invertible elements) free submonoids generated
by any set of pairs of relatively prime non-invertible square-free elements.

\medskip

It also turned out that the condition
\begin{itemize}
	\item[(3) ] $\Irr M\subset\Sqf H$
\end{itemize}
is equivalent to the conjunction of (1) and the sentence:
\begin{itemize}
	\item[(4) ] for every $a$, $b\in M$, if $a\rpr_M b$, then $a\rpr_H b$.
\end{itemize}

We have a transparent answer to the question of when the condition (1) be equivalent to the condition (3).

\medskip

A very important step in the conducted research was finding a factorial condition implicating the condition (3): 
\begin{itemize} 
	\item[(5) ] for every $a\in H$, $b\in\Sqf H$, if
	$a^2b\in M$, then $a$, $ab\in M$.
\end{itemize} 

\medskip 

A natural question arose, is it a necessary condition.
The answer is negative -- a counterexample was found (\cite{3}, Example 4.2). The factorial condition to (3) is interesting, five equivalent conditions were obtained (\cite{3}, Theorem 4.3),  including closure with respect to the factorization of 2s (See section \ref{R3.1}).

\medskip

Conditions (1) and (3) are related to the assumption found in the famous Jacobian conjecture.

\begin{con}
	Let $k$ be a field of characteristic $0$. For every polynomials $f_1$, $f_2$, $\dots$, $f_n\in k[x_1, \dots, x_n]$ with $n>1$, if 
	$$jac(f_1, f_2, \dots, f_n)\in k\setminus\{0\},$$ then 
	$$k[f_1, \dots, f_n]=k[x_1, \dots, x_n].$$
\end{con}

Recall a generalization of the Jacobian conjecture formulated in \cite{JZanalogs}.

\begin{con}
	\label{c}
	Let $k$ be a field characteristic $0$. For every polynomials $f_1$, $f_2$, $\dots$, $f_r\in k[x_1, \dots, x_n]$ with $n>1$ and $r\in\{2, \dots, n\}$, if
	$$\gcd(jac^{f_1, f_2, \dots, f_r}_{x_{j_1}, x_{j_2}, \dots, x_{j_r}}, 1\leqslant j_1<\dots < j_r\leqslant n)\in k\setminus\{0\},$$ 
	then 
	\begin{center} 
	$k[f_1, \dots, f_r]$ is algebraically closed in $k[x_1, \dots, x_n]$.
	\end{center} 
\end{con}

Under the assumption that $f_1$, $f_2$, $\dots$, $f_r$ are algebraically independent over $k$, the generalized Jacobian condition (assumption of Conjecture \ref{c}) is equivalent to any of the following ones (\cite{JZanalogs}):

\medskip

\begin{itemize} 
	\item[(6) ] every irreducible of $k[f_1, \dots, f_r]$ is square-free in $k[x_1, \dots, x_n]$,
	
	\item[(7) ] every square-free of $k[f_1, \dots, f_r]$ is square-free in $k[x_1, \dots, x_n]$.
\end{itemize}

Conditions (1) and (3) are a generalization of conditions (6) and (7) and therefore we call them the analogs of the Jacobian conditions.

\medskip

A side effect of the presented approach was a natural question about general relationships between square-free factorizations in different classes of monoids.
Of course, these factorizations for rings of polynomials are commonly known, and it is clear that their existence and uniqueness occur in domains with uniqueness of distribution, so e.g. certain properties hold in $\GCD$-domains.
However, these relationships have not been studied so far.
In this paper we will consider pre-Schreier monoids, $\GCD$-monoids, GCDs-monoids, ACCP-monoids, atomic monoids.

\medskip

Recall that a monoid is called $\GCD$-monoid, if for any two elements there is a greatest common divisor.
A monoid $H$ is called GCDs-monoid, if there is greatest common divisor for any subset of $H$.
A monoid $H$ is called a pre-Schreier monoid, if for any $a\in H$ the condtion is met, that for any $b$, $c\in H$ such that $a\mid bc$ there are $a_1$, $a_2\in H$ such that $a=a_1a_2$, $a_1\mid b$ and $a_2\mid c$.
A monoid $H$ is called atomic, if every non-invertible element $a\in H$ be a finite product of irreducibles (atoms).
A monoid $H$ is factorial, if for any non-invertible $a\in H$ an element $a$ we can presented in the form product of irreducibles and  $a=q_1q_2\dots q_k=r_1r_2\dots r_l$, where $q_1, q_2, \dots, q_k, r_1, r_2, \dots, r_l\in\Irr H$ implies $k=l$ and there is a permutation $\sigma$ such that $q_1\sim r_{\sigma(1)}$, $q_2\sim r_{\sigma(2)}$, $\dots$, $q_k\sim r_{\sigma(k)}$.
A monoid $H$ is called ACCP-monoid any ascending sequence principal ideals of $H$ stabilizes.

\medskip

In section \ref{R3.4} we examine the dependencies between square-free factorizations, conditions of existence of certain square-free divisors, and between square-free factorizations and conditions of existence of certain square-free divisors.
The conditions for the existence of certain square-free divisors result from the appropriate factorization, and the condition for the existence of a square-free divisor in a square plays an important role in reasoning about the inclusions (1) and (3).

\medskip

In this context, the concept of a radical generator is very important introduced by A.~Reinhart in 2012 in \cite{Reinhart}.
The element of monoid is called radical if the principal ideal is generated by this element be a radical ideal.
The set of all radical generators of a monoid $H$ will be denoted by $\Gpr H$.
Reinhart's explores the properties of radically factorial monoids, i.e. such that each element is a product of radical generators.
He does not consider various types of radical factorization, nor relationships with square-free factorization.
Let us add that the property of the radical generator (although the author does not use this name) appeared in the work of G. Angerm\"uller published in 2017 in the Grauert-Remmert normality criterion (\cite{Angermuller}, Proposition 31).

\medskip

The radical generator is square-free, so radical factorizations are square-free factorizations.
Therefore, in the section \ref {R3.4} we also study general relationships between radical factorizations, conditions of existence of certain radical divisors, as well as between factorizations and conditions of existence of some divisors (square-free or radical).

\medskip

In \cite{3} the relationship between the four square-free factorizations and two conditions for the existence of square-free divisors was investigated. In this paper, I present the latest results, which include the dependencies binding the next three conditions for the existence of square-free divisors and a total of nine factorization and conditions for the existence of radical divisors.

\medskip

Let's define another class of monoid. A monoid $H$ is called SR-monoid, if $\Gpr H=\Sqf H$.

It turns out that considering the obtained dependencies, we can consider various ways of classifying monoids due to square-free and radical factorization and due to the conditions of existence of certain square-free or radical divisors. The potential number of cases is: 7, 11, 26, 57, 324, 2708, 2960. These numbers depend on the given monoid property (GCDs, $\GCD$, pre-Schreier, SR, atomicity, ACCP, general, respectively).
The results of these studies are described in section \ref{R4.1}.

\section{Auxiliary statements}

In this section we present the Lemmas that we will need later in the next paper.

\begin{lm}
	\label{l2.1.5}
	Let $H$ be a monoid.
	
	\noindent
	$(a)$ 
	Let $a\in\Sqf H$ and $b\in H$.
 	If $b\mid a$ then $b\in\Sqf H$.
	
	\noindent
	$(b)$ Let $a\in\Gpr H$ and $b\in H$.
	If $b\mid a$ then $b\in\Gpr H$.
\end{lm}

\begin{proof}
	(a) 
	Suppose $b\notin\Sqf H$. Then there exists $d\in H\setminus H^{\ast}$ such that $d^2\mid b$. Hence $d^2\mid a$. A contradiction.
	
	\medskip
	
	(b)
	\cite{3}, Lemma 6.2. 
\end{proof}

\begin{lm}
	\label{l2.1.6}
	Let $H$ be a monoid. If $a\in\Sqf H$ and $a=b_1b_2\ldots b_n$, then $b_i\rpr b_j$ for $i, j\in\{1, \ldots , n\}$, $i\neq j$.
\end{lm}

\begin{proof}
\cite{3}, Lemma 3.1.
\end{proof}

\begin{lm}
	\label{l2.2.5}
	Let $H$ be a pre-Schreier monoid.
	
	\medskip
	\noindent
	(a) 
	Let $a,b,c,d\in H$.
	If $ab=cd$, $a\:\rpr\,c$ and $b\:\rpr\,d$, then $a\sim d$
	and $b\sim c$.
	
	\medskip
	\noindent
	(b) 
	Let $a_1, a_2, \dots, a_n, b\in H$.
	If $a_i\:\rpr\,b$ for $i=1, 2, \dots, n$, then $a_1a_2\ldots a_n\:\rpr\,b$.
	
	\medskip
	\noindent
	(c)
	Let $a,b\in H$. If $a\rpr b$, then $a^k\rpr b^l$ for any $k,l\in\mathbb{N}$.
	
	\medskip
	\noindent
	(d) 
	Let $a_1, a_2, \dots, a_n\in H$.
	If $a_1, a_2, \dots, a_n\in\Sqf H$ and $a_i\:\rpr\,a_j$ for $i, j\in\{1, 2, \ldots, n\}$, $i\neq j$,
	then $a_1a_2\ldots a_n\in\Sqf H$.
	
	\medskip
	\noindent
	(e) 
	Let $a_1, a_2, \dots, a_n\in\Sqf H$, $b\in H$.
	If $a_i\rpr a_j$ for $i, j\in\{1, 2, \ldots, n\}$, $i\neq j$ and $a_i\mid b$ for $i=1, 2, \dots, n$,
	then $a_1a_2\ldots a_n\mid b$.
\end{lm}

\begin{proof}
	(a), (d) 
	The proof is similar to \cite{2}, Lemma 2 (b), (e).
	
	\medskip
		
	\noindent
	(b), (e) \cite{3}, Lemma 6.3 (b), (d).
	
	\medskip
	
	\noindent
	(c)
	Let $a,b\in H$. Assume $a\rpr b$. Then by (b) we get $a^k\rpr b$ for any $k\in\mathbb{N}$. And again by (b) we have $a^k\rpr b^l$ for any $l\in\mathbb{N}$.
\end{proof}

In the following Proposition we have a very important property in a pre-Schreier monoid.

\begin{pr}
	\label{p2.2.6}
	Let $H$ be a pre-Schreier monoid. Then
	$$\Gpr H=\Sqf H.$$
\end{pr}

\begin{proof}
	\cite{3}, Proposition 6.4.
\end{proof}

\begin{lm}
	\label{l2.1.8}
	Let $H$ be a GCDs-monoid and $a\in H$. Let $X\subset H$ be any non-empty subset of set of divisors of $a$. Then there is $\GCD(X)$.
\end{lm}

\begin{proof}
	Let $Y=\{d\in H\mid \exists c\in X: a=cd\}$. Denote by $e$ a greatest common divisor of the set $Y$. Then $e$ divides every element of the set $Y$, so by definition of $Y$ we get $e\mid a$. We have $a=ef$, where $f\in H$.  
	We will show $f=\GCD(X)$. 
	
	\medskip
	
	First we prove that $f$ is least common multiple of elements of the set $X$.
	Consider any element $c\in X$. Since $c\mid a$, then $a=cd$, where $d\in H$. We have $d\in Y$, so $d=eg$, where $g\in H$. Thus, since $d=eg$, then $cd=ceg$, and since $ef=a=cd$, then $ef=ceg$. Then $f=cg$, so $c\mid f$.
	
	\medskip
	
	Now, we will show that every least common multiple of elements of $X$ is the multiple of element $f$. Consider any element $c\in X$ such that $a=cd, d\in Y$. We know that $c\mid h$, so $cd\mid hd$, hence $a\mid hd$. Let $Z=\{ bh, b\in Y\}$. Then we have $\GCD(Z)=h\GCD(Y)=he$. Since $a\mid hl$, then $a\mid eh$. We know $a=ef$, hence $ef\mid eh$, so $f\mid h$. 
\end{proof}

\begin{lm}
	\label{l2.1.9}
	Let $H$ be a monoid and $X\subset\Gpr H$. Assume that there is $\GCD(X)$. Then $\LCM(X)\in\Gpr H$.
\end{lm}

\begin{proof}
	Denote $l=\LCM(X)$. Consider any element $b\in H$ such that $l\mid b^n$ for some $n\in\mathbb{N}$. Since $l$ is the least common multiple of set $X$, then for any $c\in X$ we have $c\mid l$. Then $c\mid b^n$. Because $c\in\Gpr H$, so $c\mid b$. Then $l\mid b$. 
\end{proof}

\section{Types of factorization and square-free or radical extraction}
\label{R3.1}

In this chapter we consider the relationship between square-free and radical factorizations and the conditions for the existence of some square-free or radical divisors in some monoids.

\medskip

The following properties of the monoid $H$ are paired: the square-free version and the radical version, for example in \circled{0s }/\circled{0r } the fragment ,,$s_1, s_2, \dots, s_n$ $\in$ $\Sqf H/\Gpr H$" we read that for property \circled{0s } we have ,,$s_1, s_2, \dots , s_n$ $\in$ $\Sqf H$", and for property \circled{0r } we have "$s_1, s_2, \dots, s_n$ $\in$ $\Gpr H$". 
In some Lemmas we also have a similar formulation in two variants denoted by $\Sqf H/\Gpr H$ and we read in the same way that the Lemma was formulated for square-free elements or for radical generators.

\medskip

Let $H$ be a monoid.
Consider the following conditions:

\medskip

\noindent
\circled{0s } / \circled{0r } 
For any $a\in H$ there are $n\in\mathbb{N}$ and $s_1,s_2,\dots,s_n\in\Sqf H/\Gpr H$ such that $$a=s_1s_2\dots s_n,$$

\medskip

\noindent
\circled{1s } / \circled{1r } \
for any $a\in H$ there are $n\in\mathbb{N}$ and $s_1,s_2,\dots,s_n\in\Sqf H/\Gpr H$ satysfying the condition $s_i\rpr s_j$ for $i, j\in\{1$, $2$, $\dots$ , $n\}$, $i\neq j$ such that $$a=s_1s_2^2s_3^3\ldots s_n^n,$$

\medskip

\noindent
\circled{2s } / \circled{2r } \
for any $a\in H$ there are $n\in\mathbb{N}$ and $s_1,s_2,\dots,s_n\in\Sqf H/\Gpr H$ satysfying the condition $s_i\mid s_{i+1}$ for $i=1,\dots,n-1$ such that $$a=s_1s_2\dots s_n,$$

\medskip

\noindent
\circled{3s } / \circled{3r } \
for any $a\in H$ there are $n\in\mathbb{N}_0$ and $s_0,s_1,\dots,s_n\in\Sqf H/\Gpr H$ such that $$a=s_0s_1^2s_2^{2^2}\ldots s_n^{2^n},$$

\medskip

\noindent
\circled{4s } / \circled{4r } \
for any $a\in H$ there are $b\in H$ and $c\in\Sqf H/\Gpr H$ satysfying the condition $b\rpr c$ such that $$a=bc$$ and there is $d\in\Sqf H/\Gpr H$ such that $d^2\mid b$ and $b\mid d^n$ for some $n\in\mathbb{N}$,

\medskip

\noindent
\circled{4's } / \circled{4'r } \
for any $a\in H$ there are $b\in H$ and $c\in\Sqf H/\Gpr H$ satysfying the condition $b\rpr c$ such that $$a=bc$$ and for any $d\in\Sqf H/\Gpr H$, if $d\mid b$ then $d^2\mid b$,

\medskip

\noindent
\circled{5s } / \circled{5r } \
for any $a\in H$ there are $b\in H$ and $c\in\Sqf H/\Gpr H$ such that $$a=bc$$ and $a\mid c^n$ for some $n\in\mathbb{N}$,

\medskip

\noindent
\circled{5's } / \circled{5'r } \
for any $a\in H$ there are $b\in H$ and $c\in\Sqf H/\Gpr H$ such that $$a=bc$$ and for any $d\in\Sqf H/\Gpr H$, if $d\mid a$ then $d\mid c$,

\medskip

\noindent
\circled{6s } / \circled{6r } \
for any $a\in H$ there are $b\in H$ and $c\in\Sqf H/\Gpr H$ such that $$a=b^2c.$$

\section{Relationships between factorizations}
\label{R3.4}

\begin{pr}
	\label{p3.4.1}
	Let $H$ be a monoid. 
	\begin{itemize}
		\item[(a) ] The following implications holds:
		$$\begin{array}{ccccccc}
		&&\circled{0s}\\
		&\mbox{\begin{psfrags}\rotatebox{-135}{$\Leftarrow$}\end{psfrags}} && \mbox{\begin{psfrags}\rotatebox{-45}{$\Leftarrow$}\end{psfrags}}\\
		\circled{1s}&\Leftarrow &\circled{2s}&\Rightarrow &\circled{3s} \\
		&& \Downarrow && \Downarrow \\
		&& \circled{5s}&&\circled{6s}
		\end{array}$$

		\item[(b) ] The following implications holds:
		$$\begin{array}{ccccccc}
		&&\circled{0r}\\
		&\mbox{\rotatebox{-135}{$\Leftarrow$}} && \mbox{\rotatebox{135}{$\Rightarrow$}} \\
		\circled{1r}&\Leftarrow &\circled{2r}&\Rightarrow &\circled{3r} \\
		&&\Downarrow && \Downarrow \\
		\circled{4r}&\Leftarrow&\circled{5r}&&\circled{6r}&& \\
		\Downarrow && \Downarrow\\
		\circled{4'r} && \circled{5'r}
		\end{array}$$

		\item[(c) ]The following implications holds:
		
		$$\begin{array}{ccccccccc}
		\circled{0r}&\circled{1r}&\circled{2r}&\circled{3r}&\circled{4r}&\circled{5r}&\circled{6r}\\
		\Downarrow &\Downarrow&\Downarrow&\Downarrow& \Downarrow &\Downarrow &\Downarrow \\
		\circled{0s}&\circled{1s}&\circled{2s}&\circled{3s}&\circled{4s}&\circled{5s}&\circled{6s}
		\end{array}$$
	\end{itemize}
\end{pr}

\begin{proof}
	\begin{itemize}
		\item[(a) ]
		\circled{2s}$\Rightarrow$\circled{1s}
		\cite{2} Proposition 1, (a), (iv)$\Rightarrow$ (vi). 
			
		\medskip
		
		\circled{2s}$\Rightarrow$\circled{3s}
		From \circled{2s}$\Rightarrow$\circled{1s} we can present an element $a$ as $a=u_1u_2^2u_3^3\dots u_n^n$, where $u_1, u_2, \dots , u_n\in\Sqf H/\Gpr H$ satysfying the condition $u_i\rpr u_j$ for $i,j\in\{1$, $2$, $\dots$ , $n\}$, $i\neq j$, where $s_{n-i+1}=s_{n-i}u_i$ for $i\in\{1, 2, \dots , n-1\}$ and $u_n=s_1$.
		Then
		$$\prod_{k=1}^nu_k^k=
		\prod_{k=1}^nu_k^{\sum_{i=0}^r c_i^{(k)} 2^i}=
		\prod_{k=1}^n\prod_{i=0}^r u_k^{c_i^{(k)} 2^i}=
		\prod_{i=0}^r\big(\prod_{k=1}^n u_k^{c_i^{(k)}}\big)^{2^i}.$$
		Denote $t_i=\prod_{k=1}^{n} u_k^{c_i^{(k)}}$ for $i=0, 1, \dots r$. 
		Because $u_i\rpr u_j$ for $i\neq j$, so from Lemma \ref{l2.1.6} we have $t_i\in\Sqf H$. 
		Therefore $a=t_0t_1^2t_2^{2^2}\dots t_r^{2^r}$, where $k=\sum_{i=0}^r c_i^{(k)}2^i$ for $k=1, 2, \dots , n$ and $c_i^{(k)}\in\{0,1\}$. 
		
		\medskip
		
		\circled{2s}$\Rightarrow$\circled{5s}
		\cite{3}, Proposition 3.4, (ii)$\Rightarrow$(v).
		
		\medskip
		
		\circled{3s}$\Rightarrow$\circled{6s}
		Obvious.		
		
		\item[(b) ] 
		
		\circled{4r}$\Rightarrow$\circled{4'r}
		Let $e\in\Gpr H$ be such that $e\mid b$. By assumption we have $b\mid d^n$, hence $e\mid d^n$, because $e\mid b$. But $e\in\Gpr H$, so from the fact that $e\mid d^n$ we have $e\mid d$, thus $e^2\mid d^2$. By assumption we have $d^2\mid b$, so $e^2\mid b$.
		
		\medskip
		
		\circled{5r}$\Rightarrow$\circled{4r}
		Let $a=bc$, where $b\in H, c\in\Gpr H$ such that $a\mid c^m$ for some $m\in\mathbb{N}$. By assumption we can $b$ presented in the form $b=de$, where $d\in H$, $e\in\Gpr H$ such that $b\mid e^k$ for some $k\in\mathbb{N}$. 
		
		\medskip
		
		Since $e\mid b$, $b\mid a$ and $a\mid c^m$, then $e\mid c^m$. But $e\in\Gpr H$, so $e\mid c$ by definition. Then $c=ef$, where $f\in H$. By Lemma \ref{l2.1.5} we refer that $f\in\Gpr H$, and from Lemma \ref{l2.1.6} we have $e\rpr f$. From equation $b=de$ we have $be=de^2$. We get $a=bef$, where $e^2\mid be$ and $be\mid e^{k+1}$.

		\medskip

		Now we will prove that $be\rpr f$.
		From divisibilities $d\mid be$, $be\mid e^{k+1}$ and $e^{k+1}\mid c^{k+1}$ we have $d\mid c^{k+1}$ and $f\mid c$, $c\mid c^{k+1}$, so $f\mid c^{k+1}$. In other hand we have $df\mid bef$, $bef\mid a$ and $a\mid c^l$ for some $l\in\mathbb{N}$, so $df\mid c^l$. Hence since $d\mid c^k$, $f\mid c^l$, $df\mid c^l$, then $d\rpr f$. And since $e\rpr f$, then $be\rpr f$.
		
		\medskip
		
		\circled{5r}$\Rightarrow$\circled{5'r}
		Let $d\in\Gpr H$ be such that $d\mid a$. Since $d\mid a$ and by assumption $a\mid c^n$, then $d\mid c^n$. Because $d\in\Gpr H$, so $d\mid c$.
		
		\item[(c) ]	
		The proof comes from the fact that any radical generator is a square-free element.
		
	\end{itemize}
\end{proof}

Recall that in a SR-monoid the concept of a square-free element coincides with the concept of a radical generator, therefore it is enough to consider square-free properties.

\begin{pr}
	\label{p3.4.1a}
	Let $H$ be a SR-monoid. Then
	\begin{itemize}
		\item[(a) ] the following implications hold:
		
		$$\begin{array}{ccccccc}
		&&\circled{0s}\\
		&\mbox{\rotatebox{-135}{$\Leftarrow$}} & \Uparrow & \mbox{\rotatebox{135}{$\Rightarrow$}} \\
		\circled{1s}&\Leftarrow &\circled{2s}&\Rightarrow &\circled{3s} \\
		&&\Downarrow && \Downarrow \\
		\circled{4s}&\Leftarrow&\circled{5s}&&\circled{6s}&& \\
		\Downarrow && \Downarrow\\
		\circled{4's} && \circled{5's}
		\end{array}$$

		\item[(b) ] the following equivalences hold:
		
		$$\begin{array}{ccccccccccccccccc}
		\circled{0r}&&\circled{1r}&&\circled{2r}&&\circled{3r}&&\circled{4r}&&\circled{4'r}&&\circled{5r}&&\circled{5'r}&&\circled{6r}\\
		\Updownarrow &&\Updownarrow&&\Updownarrow&&\Updownarrow&&\Updownarrow&&\Updownarrow&& \Updownarrow &&\Updownarrow &&\Updownarrow \\
		\circled{0s}&&\circled{1s}&&\circled{2s}&&\circled{3s}&&\circled{4s}&&\circled{4's}&&\circled{5s}&&\circled{5's}&&\circled{6s}
		\end{array}$$
	\end{itemize}
\end{pr}

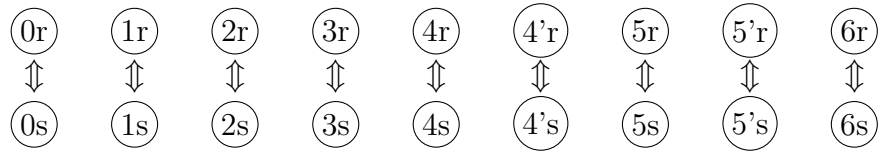
\begin{proof}
	\begin{itemize}
		\item[(a) ]
		Since $H$ is a SR-monoid, so every implications from Proposition \ref{p3.4.1} (b) hold.
		
		\item[(b) ]	
		Obvious.
	\end{itemize}
\end{proof}

Since in pre-Schreier monoids, $\GCD$-monoids and GCDs-monoids the SR property holds, therefore in the following three Propositions it is enough to consider square-free dependencies.

\begin{pr}
	\label{p3.4.3}
	Let $H$ be a pre-Schreier monoid. Then
	
	\begin{itemize} 
		
		\item[(a) ] the following implications and equivalences hold:
		
		$$\begin{array}{ccccccc}
		&&\circled{0s}\\
		&\mbox{\rotatebox{-135}{$\Leftarrow$}} &&
		\mbox{\rotatebox{-45}{$\Leftarrow$}} \\
		\circled{1s}&\Leftrightarrow &\circled{2s}&\Rightarrow&\circled{3s} \\
		\Downarrow && \Downarrow && \Downarrow \\
		\circled{4s}&\Leftrightarrow&\circled{5s}&& \circled{6s} \\
		\Downarrow && \Downarrow\\
		\circled{4's} && \circled{5's}
		\end{array}$$
		
		\item[(b) ] if the condition \circled{2s} holds, then $H$ be GCD-monoid.
	\end{itemize}
\end{pr}

\begin{proof}
	\begin{itemize} 
		\item[(a) ]
		\circled{1s}$\Rightarrow$\circled{2s}
		\cite{2}, Proposition 1, (b), (vi)$\Rightarrow$ (iv).
		
		\medskip
		
		\circled{1s}$\Rightarrow$\circled{4s}
		Put $b=s_2^2s_3^3\dots s_n^n$ and $c=s_1$. From the fact that $s_1, s_2, \dots , s_n$ are pairwise relatively prime results $b\rpr c$ from Lemma \ref{l2.2.5} (d). Moreover for $d=s_2s_3\dots s_n$ we have $d^2\mid b, b\mid d^n$. Because $s_i\rpr s_j$ for $i, j\in\{2, 3, \ldots , n\}$, $i\neq j$, so from Lemma \ref{l2.2.5} (e) we have $d\in\Sqf H$.
		
		\medskip
		
		\circled{4s}$\Rightarrow$\circled{5s}
		Assume $a=bc$, where $b\in H$, $c\in\Sqf H$ such that $b\rpr c$ and $b=d^2e$, $b\mid d^m$, where $d\in\Sqf H$ and $m\in\mathbb{N}$. Then $a=d^2ec=(de)(cd)$. Since $d\mid b, b\rpr c$, then $d\rpr c$, so $cd\in\Sqf H$ by Lemma \ref{l2.2.5} (d). We get also that since $b\mid d^m$, then $bc\mid d^mc$, and because $d^mc\mid (cd)^m$, so $a\mid (cd)^m$. 
		
		\medskip
		
		The other implications hold from Proposition \ref{p3.4.1}.
		
		\item[(b) ] 
		\cite{3}, Reviewer's remark, p. 865.
	\end{itemize}
\end{proof}

\begin{pr}
	\label{p3.4.4}
	Let $H$ be a GCD-monoid. Then the following implications and equivalences hold:
	$$\begin{array}{ccccccc}
	&&\circled{0s}\\
	&\mbox{\rotatebox{-135}{$\Leftrightarrow$}} &\Updownarrow&
	\mbox{\rotatebox{-45}{$\Leftrightarrow$}} \\
	\circled{1s}&\Leftrightarrow &\circled{2s}&\Leftrightarrow&\circled{3s} \\
	\Downarrow &&\Downarrow && \Downarrow \\
	\circled{4s}&\Leftrightarrow&\circled{5s}&&\circled{6s} \\
	\Downarrow && \Downarrow\\
	\circled{4's} && \circled{5's}
	\end{array}$$
\end{pr}

\begin{proof}
	\circled{1s}$\Leftrightarrow$\circled{2s}$\Leftrightarrow$\circled{3s}
	\cite{2}, Proposition 1 (b).
	
	\medskip
	
	\circled{0s}$\Rightarrow$\circled{2s}
	\cite{3}, Reviewer's remark, p. 854.
	
	\medskip
	
	The other implications and equivalences hold from Proposition \ref{p3.4.3}.
	
\end{proof}

\begin{pr}
	\label{p3.4.5}
	Let $H$ be a GCDs-monoid. Then the condition $\circled{5's}$ holds.
\end{pr}

\begin{proof}
	Let $a\in H$ and $X=\{d\in\Sqf H; d\mid a\}$.
	From Lemma \ref{l2.1.8} there exists $\LCM (X)$. Let $c=\LCM (X)$. By Lemma \ref{l2.1.9} we get that $c\in\Sqf H$. Since every element belonging to $X$ divides $a$, the $c\mid a$. Hence $a=bc$ for some $b\in H$. Consider any $d\in\Sqf H$ such that $d\mid a$. But $d\in X$, hence $d\mid c$, because $c=\LCM (X)$.   
\end{proof}

\medskip

Note that in an atomic monoid the \circled{0s} property holds.

\medskip

\begin{pr}
	\label{p3.4.9}
	Let $H$ be an ACCP-monoid. Then  
	
	\begin{itemize}
		\item[(a) ]
		the conditions $\circled{0s}$, $\circled{3s}$ and $\circled{6s}$ hold,

		\item[(b) ] the following implicatios and equivalences hold:
		
		$$\begin{array}{ccccccccccccccc}
		&&&&&&&&&&\circled{5'r} && \circled{4'r}\\
		&&&&&&&&&&\Uparrow && \Uparrow\\
		\circled{0r}&\Leftarrow &\circled{1r}&\Leftarrow&\circled{2r}&\Leftrightarrow&\circled{3r}&\Leftrightarrow&\circled{6r}&\Leftrightarrow&\circled{5r}&\Rightarrow&\circled{4r} \\
		&&\Downarrow && \Downarrow &&&&&& \Downarrow && \Downarrow \\
		&&\circled{1s}&&\circled{2s}&&&\Rightarrow&&&\circled{5s}&&\circled{4s}\\
		&&&& \Uparrow \\
		&&&& \circled{5's}
		\end{array}$$
		
	\end{itemize}
\end{pr}

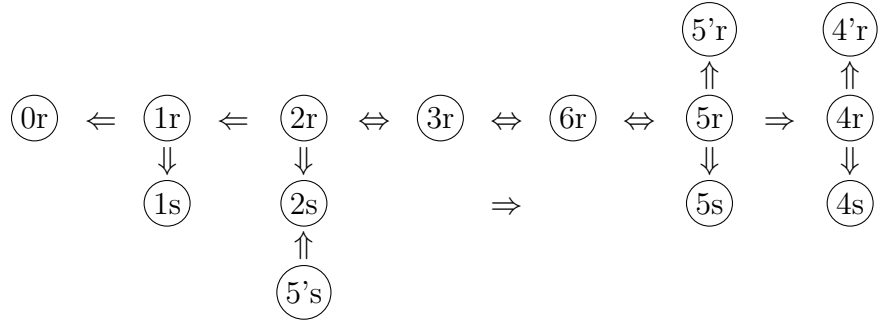
\begin{proof}
	\begin{itemize}
		\item[(a) ]
		
		\cite{2}, Proposition 1 (c), (i), (iii).

\medskip

\item[(b) ]
\circled{5's}$\Rightarrow$\circled{2s}
Consider any element $a\in H$. We can presented element $a$ in the form $a=b_1c_1$, where $b_1\in H$, $c_1\in\Sqf H$ and for every $d\in\Sqf H$, if $d\mid a$, then $d\mid c$.

\medskip

We can presented element $b_1$ in the form $b_1=b_2c_2$, where $b_2\in H$, $c_2\in\Sqf H$ and for every $d\in\Sqf H$, if $d\mid b_1$, then $d\mid c_2$.

\medskip

An element $b_2$ we can presented in the form $b_2=b_3c_3$, where $b_3\in H$, $c_3\in\Sqf H$ and for every $d\in\Sqf H$, if $d\mid b_2$, then $d\mid c_3$.

\medskip

Continuing, we get an ascending sequence of principal ideals
$$(b_1)\subset (b_2)\subset (b_3)\subset\dots .$$ 
Then by ACCP condition there exists $m\in\mathbb{N}$ such that $$(b_n)=(b_{n+1})=(b_{n+2})\dots .$$ In particular $(b_k)=(b_{k+1})$, so $b_k\sim b_{k+1}$. Because $b_k=b_{k+1}c_{k+1}$, hence $c_{k+1}\in H^{\ast}$. we know that for any element $d\in\Sqf H$, if $d\mid b_k$, then $d\mid c_{k+1}$. But $c_{k+1}\in H^{\ast}$, hence since $d\mid b_k$, then $d\in H^{\ast}$.

\medskip

We have 
$$a=b_1c_1=b_2c_2c_1=\dots =b_kc_kc_{k-1}\dots c_1=c_kc_{k-1}\dots c_1,$$ 
because $b_k\in H^{\ast}$. We show that for every $i=2, 3, \dots k$ the divisibiity $c_i\mid c_{i-1}$ holds. For $i=2$ we have $c_2\mid b_1$, because $b_1=b_2c_2$. Since $c_2\mid b_1$, then $c_2\mid a$. Then by the assumption $c_2\mid c_1$. For $i= 3, 4, \dots $ we know that for every element $b_{i-1}$ we can presented in the form $b_{i-1}=b_ic_i$, hence $c_i\mid b_{i-1}$. We also know that $b_{i-1}\mid b_{i-2}$. And hence $c_i\mid b_{i-2}$. By the assumption we have for any element $d\in\Sqf H$, if $d\mid b_{i-2}$, then $d\mid c_{i-1}$, so since $c_i\mid b_{i-2}$, then $c_i\mid c_{i-1}$, because $c_i\in\Sqf H$.

\medskip

\circled{6s}$\Rightarrow$\circled{3s}
Consider any element $a\in H$. The element $a$ can be presented in the form $a=b_1^2c_1$, where $b_1\in H$, $c_1\in\Sqf H/\Gpr H$. 

\medskip

An element $b_1$ can be presented in the form $b_1=b_2^2c_2$, where $b_2\in H$, $c_2\in\Sqf H/\Gpr H$. Similarly, we can presented an element $b_2$ in the form $b_2=b_3^2c_3$, where $b_3\in H$, $c_3\in\Sqf H/\Gpr H$.

\medskip

By continuing this process, we obtain an ascending sequence of principal ideals
$$(b_1)\subset (b_2)\subset (b_3)\subset$$

By ACCP condition there exists $k\in\mathbb{N}$ such that $b_k\sim b_{k+1}$. And because $b_k=b_{k+1}^2c_{k+1}$, hence $b_{k+1}, c_{k+1}\in H^{\ast}$. Since $b_k\sim b_{k+1}$ and $b_{k+1}\in H^{\ast}$, then $b_k\in H^{\ast}$. 

\medskip

Then
\begin{eqnarray*} 
a&=b_1^2c_1=b_2^{2^2}c_2^2c_1=b_3^{2^3}c_3^{2^2}c_2^2c_1=\dots = b_k^{2^k}c_k^{2^{k-1}}c_{k-1}^{2^{k-2}}\dots c_2^2c_1=\\
&=s_0s_1^2s_2^{2^2}\dots s_n^{2^n},
\end{eqnarray*}
where $s_0=c_1$, $s_1=c_2$, $s_2=c_3$, $\dots$, $s_{n-1}=c_k$, $s_n=b_k$.

\medskip
		
\item[(c) ]
\circled{5r}$\Rightarrow$\circled{2r}

Consider any element $a\in H$. We can introduced the element $a$ in the form $a=b_1c_1$, where $b_1\in H$, $c_1\in\Gpr H$ and $a\mid c_1^{n_1}$ holds for some $n_1\in\mathbb{N}$. 

\medskip

An element $b_1$ can be presented in the form $b_1=b_2c_2$, where $b_2\in H$, $c_2\in\Gpr H$ and $b_1\mid c_2^{n_2}$ holds form some $n_2\in\mathbb{N}$. 

\medskip

An element $b_2$ can be presented in the form $b_2=b_3c_3$, where $b_3\in H$, $c_3\in\Gpr H$ and $b_2\mid c_3^{n_3}$ holds for some $n_3\in\mathbb{N}$.

\medskip

Continuing our reasoning we get an increasing sequence of principal ideals
$$(b_1)\subset (b_2)\subset (b_3)\subset .$$

By ACCP condition there exists $n$ such that 
$$(b_n)=(b_{n+1})=(b_{n+2})=\dots.$$
In particular $(b_n)=(b_{n+1})$, so $b_n\sim b_{n+1}$.
And because $b_n=b_{n+1}c_{n+1}$, so $c_{n+1}\in H^{\ast}$. There is also divisibility $b_n\mid c_{n+1}^{m_{n+1}}$, hence $b_n\in H^{\ast}$.

\medskip

Then we get
$$a=b_1c_1=b_2c_2c_1=b_3c_3c_2c_1=\dots =b_nc_nc_{n-1}\dots c_2c_1=s_1s_2\dots s_n,$$
where $s_1=b_nc_n$, $s_2=c_{n-1}$, $s_3=c_{n-2}$, $\dots$, $s_n=c_1$.

\medskip

It remained to prove that for $i=1, 2, \dots, n-1$ the condition $c_{i+1}\mid c_i$ holds.
For $i=1$ we have divisibilities $c_2\mid b_1, b_1$, $b_1\mid a$, $a\mid c_1^{m_1}$, hence $c_2\mid c_1$, because $c_2\in\Gpr H$.
For $i>1$ divisibilities $c_{i+1}\mid b_i$, $b_i\mid b_{i-1}$, $b_{i-1}\mid c_i^{m_i}$ holds, and hence $c_{i+1}\mid c_i$. Since $c_{i+1}\in\Gpr H$, then $c_{i+1}\mid c_i$.

\medskip
		
\circled{6r}$\Rightarrow$\circled{5r}
Consider any element $a\in H$.
An element $a\in H$ can be presented in the form $a=b_1^2c_1$, where $b_1\in H$, $c_1\in\Gpr H$.
		
\medskip
		
An element $b_1c_1$ can be presented in the form $b_1c_1=b_2^2c_2$, where $b_2\in H$, $c_2\in\Gpr H$. Similarly, we can presented an element $b_2c_2$ in the form $b_2c_2=b_3^2c_3$, where $b_3\in H$, $c_3\in\Gpr H$. 
		
\medskip

By repeating the process, we obtain the following ascending sequence of principal ideals 
$$(b_1c_1)\subset (b_2c_2)\subset (b_3c_3) \dots $$ 
By ACCP condition there exists $k\in\mathbb{N}$ such that $$(b_kc_k)=(b_{k+1}c_{k+1})=(b_{k+2}c_{k+2})\dots $$ 
In particular $(b_kc_k)=(b_{k+1}c_{k+1})$, so $b_kc_k\sim b_{k+1}c_{k+1}$. 
From the equation $b_kc_k=b_{k+1}^2c_{k+1}$ and from $b_kc_k\sim b_{k+1}c_{k+1}$ we get $b_{k+1}\in H^{\ast}$. 

\medskip

We have the following divisibility: 
$$c_{k+1}\mid b_kc_k, b_kc_k\mid b_{k-1}c_{k-1}, \dots, b_2c_2\mid b_1c_1, b_1c_1\mid a.$$
Therefore, since $a=b_1^2c_1$, then $a\mid (b_1c_1)^2$. Since $b_1c_1=b_2^2c_2$, then $b_1c_1\mid (b_2c_2)^2$. Generally for $i=2, 3, \dots, k$ we have $b_{k-1}c_{k-1}\mid (b_kc_k)^2$. Hence $a\mid (b_kc_k)^{2^k}$. Since $b_kc_k\sim c_{k+1}$, then $a\mid c_{k+1}^{2^k}$. 
	
\end{itemize}

\medskip

The other implications hold from Proposition \ref{p3.4.1}.
\end{proof}

\section{Unique representation}

In this section, we present the unique presentation of the factorization and the conditions of existence of square-free and radical divisors.

\begin{pr}
	\label{p3.6.1}
	Let $H$ be a monoid.
		

	Consider any elements $a,c\in H$, $b,d\in \Gpr H$, such that for any $e\in\Gpr H$ implications hold:
	if $e\mid ab$, then $e\mid b$ and if $e\mid cd$, then $e\mid d$.
	If
	$$ab\sim cd,$$
	then $a\sim c$ and $b\sim d$.

\end{pr}

\begin{proof}
Assume $ab\sim cd$.
We see that $b\in\Gpr H$ and $b\mid cd$, so $b\mid d$ by assumption.
Similarly, we justify divisibility $d\mid b$. Hence $b\sim d$, and then $a\sim c$.
	
\end{proof}

The uniques of \circled{2r}, \circled{5r} was proved in \cite{3}, Proposition 6.5 (a), (b).



The uniques of \circled{1s} was proved in \cite{3}, Proposition 6.6.

\begin{pr}
	\label{p3.6.3}
	Let $H$ be a GCD-monoid.
	
	\textit{(a) } 
	Consider any elements $a, c\in H$, $b, d\in\Sqf H$, such that $a\rpr b$, $c\rpr d$ and for some elements $e, f\in\Sqf H$ and $m, n\in\mathbb{N}$ divisibilities $e^2\mid a$, $a\mid e^m$ and $f^2\mid c$, $c\mid f^n$ hold. If
	$$ab\sim cd,$$
	then $a\sim c$, $b\sim d$.
	
	\medskip
	
	\textit{(b) } 
	Consider any elements $a, c\in H$, $b, d\in\Sqf H$, such that $a\rpr b$, $c\rpr d$ and for any $g\in\Sqf H$ the implication holds: if $g\mid a$, then $g^2\mid a$. If
	$$ab\sim cd,$$
	then $a\sim c$, $b\sim d$.
	
	\medskip
	
	\textit{(c) } 
	Consider any elements $a,c\in H$ and $b,d\in\Sqf H$. If $$a^2b=c^2d,$$ then $a\sim c$ and $b\sim d$.
	
	\medskip
	
	\textit{(d) } 
	Consider any elements $s_0, s_1, \dots, s_n\in\Sqf H$ and $t_0, t_1, \dots, t_n\in\Sqf H$. If $$s_n^{2^n}s_{n-1}^{2^{n-1}}\dots s_1^2s_0=
	t_n^{2^n}t_{n-1}^{2^{n-1}}\dots t_1^2t_0,$$
	then $s_i\sim t_i$ for $i=0, 1, \dots, n$.
\end{pr}

\begin{proof}
 (a)
	Assume $ab\sim cd$. Put $g=\GCD (d,e)$. Since $d\in\Sqf H$, then by Lemma \ref{l2.1.5} we have $g\in\Sqf H$, because $g\mid d$. Since $g\mid e$, then $g^2\mid e^2$, and hence $g^2\mid a$, because $e^2\mid a$. Since $g^2\mid a$ and $a\mid cd$, so $g^2\mid cd$. 
	Let us remind $g\mid d$, then $g^2\mid d^2$. Since $g^2\mid cd$, $g^2\mid d^2$ and $c\rpr d$, hence by Lemma we refer $g^2\mid\GCD(cd, d^2)$, so $g^2\mid d$.
	Because $d\in\Sqf H$, so $g\in H^{\ast}$.
	Then $d\rpr e$, because $g$ is their greatest common divisor. Therefore by Lemma \ref{l2.2.5} (c) we refer $d\rpr e^m$, and hence $d\rpr a$, because $a\mid e^m$. Similarly, we justify that $b\rpr c$ putting $h=\GCD (b, f)$ and we repeat the reasoning. Then by Lemma \ref{l2.2.5} (a) we have $a\sim c$, $b\sim d$.
	
	\medskip
	
(b)
	Assume $ab\sim cd$. Put $g=\GCD (a,d)$. Since $d\in\Sqf H$, then by Lemma \ref{l2.1.5} we have $g\in\Sqf H$, because $g\mid d$. Since $g\mid a$, then $g^2\mid a$ by the assumption. Hence $g^2\mid cd$. Let us remind $g\mid d$, then $g^2\mid d^2$. Since $g^2\mid cd$, $g^2\mid d^2$ and $c\rpr d$, hence we refer $g^2\mid\GCD(cd, d^2)$, so $g^2\mid d$.
	Because $d\in\Sqf H$, so $g\in H^{\ast}$.
	Then $a\rpr d$, because $g$ is their greatest common divisor. Because $d\mid ab$, hence $d\mid b$. Similarly, we justify that $b\rpr c$ putting $h=\GCD (b, c)$ and we repeat the reasoning. Then by Lemma \ref{l2.2.5} (a) we have $a\sim c$, $b\sim d$.
			
	\medskip
	
(c), (d)	The uniques of \circled{6s}, \circled{2s} was proved in \cite{2}, Proposition 2 (i), (ii).
\end{proof}

\section{Some examples}

\begin{ex}
	\label{exex}
	Let $$H=\mathbb{N}_{\geqslant k}\cup\{0\}.$$
	$H$ be a $\GCD$-monoid. All conditions are met from \ref{R3.1}.
\end{ex}

\bigskip

\begin{ex}
	For the established $k\in\mathbb{N}$, let $H=\mathbb{Q}_{\geqslant k}\cup\{0\}$. 
	A monoid $H$ does not meet all the conditions in \ref{R3.1}.
\end{ex} 

\bigskip

\begin{ex}
	Let $H=\mathbb{N}_0^2$ with addition action. 
	$H$ be an ACCP-monoid. 
	A monoid $H$ satisfies the conditions: 0s, 1s, 2s, 3s, 4s, 4'r, 5s, 5'r, 6s.
	$H$ does not meet the conditions: 0r, 1r, 2r, 3r, 4r, 4's, 5r, 5's, 6r.
\end{ex}

\bigskip

\begin{ex}
	Let $H$ be a monoid, not a group such that every element of $H$ be a square.
	In particular $\mathbb{Q}_{\geqslant 0}$ and $\big\langle \dfrac{1}{2^n}\mid n\in\mathbb{N}\big\rangle$.
	The monoid $H$ satisfies the conditions: 4's, 4'r, 5's, 5'r, 6s, 6r. The others are not met.
\end{ex}

\bigskip

\begin{ex} 
	Consider a submonoid of free monoid $$H=\langle x_1, x_2, \dots, y_1, y_2, \dots \mid y_i=x_{i+1}^py_{i+1}^q, i=1, 2, \dots \rangle$$ for any $p$, $q\in\mathbb{N}$.
	$H$ be a $\GCD$-monoid, not ACCP-monoid. 
	Consider the special cases of the monoid $H$.
	
	\begin{itemize} 
			\item[(1) ]
			Let $$H=\langle x_1, x_2, \dots, y_1, y_2, \dots\mid y_i=x_{i+1}y_{i+1}, i=1, 2, \dots \rangle.$$
			All conditions are met from \ref{R3.1}.
	\end{itemize} 
	
	\begin{itemize} 
		\item[(2) ]
		Let $q=2r$.
		Then
		$$H=\langle x_1, x_2, \dots, y_1, y_2, \dots\mid y_i=x_{i+1}^py_{i+1}^{2r}, i=1, 2, \dots \rangle.$$
		The monoid $H$ satisfies the conditions: 4s, 4's, 5s, 5's. The others are not met.
	\end{itemize} 
		
	\begin{itemize} 
			\item[(3) ]
			Let $q=2t+1$.
			Then
			$$H=\langle x_1, x_2, \dots, y_1, y_2, \dots \mid y_i=x_{i+1}^py_{i+1}^{2t+1}, i=1, 2, \dots, p\neq 1, t\neq 0\rangle.$$
			The monoid $H$ satisfies the conditions: 4s, 4's, 5s, 5's. The others are not met.
	\end{itemize} 
\end{ex}

\bigskip

\begin{ex}
	Let
	$$H=\langle x_1, x_2, \dots, y_1, y_2, \dots, z_1, z_2, \dots \mid x_{i+1}=x_i^2y_i, y_{i+1}=y_iz_i, i=1, 2,\dots \rangle.$$
	A monoid $H$ does not satisfies 1s, 1r, 2s, 2r.
	Other conditions have not been investigated.
\end{ex}

\section{Classifications of monoids with respect to square-free and radical	factorizations}
\label{R4.1}

In section \ref{R3.1} determined $18$ properties of monoids:

\begin{center}
	0s, 0r, 1s, 1r, 2s, 2r, 3s, 3r, 4s, 4r, 4's, 4'r, 5s, 5r, 5's, 5'r, 6s, 6r.
\end{center}

We can treat these properties as propositional forms defined on the class of monoids.

\medskip

To simplify the reasoning, it is worth considering the following pairs:

\begin{center}
	0sr=(0s,0r), 1sr=(1s,1r), 2sr=(2s,2r), 3sr=(3s,3r), 4sr=(4s,4r), 5sr=(5s,5r), 6sr=(6s,6r). 
\end{center}

For A$=0, 1, 2, 3, 4, 5, 6$ as the logical value of the pair Asr we take the sum of the logical values of the forms As and Ar:
\begin{center}
	\vvv(Asr) $=$ \vvv(As) $+$ \vvv(Ar),
\end{center}

\noindent
where \vvv~~be the logical value of the monoid property.

\medskip

Note that there is an implication of Ar $\Rightarrow$ As, so \vvv(Asr) uniquely defines a pair of values (\vvv(As), \vvv(Ar)).

\medskip

In the tables below, we present the logical values of the entered sentence form pairs.

$$\begin{tabular}{cccc}
\begin{tabular}{|c|c||c|}
\hline
\vvv(0s) & \vvv(0r) & \vvv(0sr) \\
\hline
\hline
1 & 1 & 2\\
\hline
1 & 0 & 1\\
\hline
0 & 0 & 0\\
\hline
\end{tabular} \,

\begin{tabular}{|c|c||c|}
	\hline
	\vvv(1s) & \vvv(1r) & \vvv(1sr) \\
	\hline
	\hline
	1 & 1 & 2\\
	\hline
	1 & 0 & 1\\
	\hline
	0 & 0 & 0\\
	\hline
\end{tabular}

$\dots$

\begin{tabular}{|c|c||c|}
	\hline
	\vvv(6s) & \vvv(6r) & \vvv(6sr) \\
	\hline
	\hline
	1 & 1 & 2\\
	\hline
	1 & 0 & 1\\
	\hline
	0 & 0 & 0\\
	\hline
\end{tabular}
\end{tabular}$$

It seems that the implications of 4'r $\Rightarrow$ 4's, 4's $\Rightarrow$ 4'r, 5'r $\Rightarrow$ 5's, 5's $\Rightarrow$ 5'r generally do not apply, so the value pair is not uniquely defined by the sum of the values. It is very important in this situation to find counterexamples.

\medskip

We will often use the following simple observation:

\medskip

\begin{center}
p $\Rightarrow$ q is exactly where \vvv(p) $\leqslant$ \vvv(q).
\end{center} 

\begin{lm}
\label{444}
For A, B $\in\{0, 1, \dots, 6\}$ such that $A\neq B$, system of implications
$$\begin{array}{ccccccccccccccc}
\circled{Ar} & \Rightarrow & \circled{Br}\\
\Downarrow && \Downarrow\\
\circled{As} & \Rightarrow & \circled{Bs}
\end{array}$$
is exactly where 
\normalfont \vvv(Asr) $\leqslant$ \vvv(Bsr).
\end{lm}

\medskip

We will determine all possible systems of logical values in the class of all monoids considering the dependencies collected in the Proposition \ref{p3.4.1}.

\medskip

We will start by determining all possible values of logical properties of square-free and radical factorization:

\begin{center}
	0s, 0r, 1s, 1r, 2s, 2r, 3s, 3r.
\end{center}

The above dependencies are presented in the diagram below.

$$\begin{array}{ccccccc}
\circled{2r} & & \Rightarrow & & \circled{1r}\\
& \mbox{\begin{psfrags}\rotatebox{135}{$\Leftarrow$}\end{psfrags}} &&& \Downarrow & \mbox{\begin{psfrags}\rotatebox{135}{$\Leftarrow$}\end{psfrags}}\\
\Downarrow&&\circled{3r} &  & \Rightarrow &  &\circled{0r}\\
&& \Downarrow && \Downarrow &&  \\
\circled{2s} & & \Rightarrow & & \circled{1s} && \Downarrow\\
& \mbox{\begin{psfrags}\rotatebox{135}{$\Leftarrow$}\end{psfrags}} & \Downarrow && & \mbox{\begin{psfrags}\rotatebox{135}{$\Leftarrow$}\end{psfrags}}& \\
&&\circled{3s} & & \Rightarrow &  &\circled{0s}\\
\end{array}$$

Note that by Lemma \ref{444} the above implications apply exactly when
\begin{center}
	\vvv(2sr) $\leqslant$ \vvv(1sr) $\leqslant$ \vvv(0sr)
	
	\smallskip 
	
	\vvv(2sr) $\leqslant$ \vvv(3sr) $\leqslant$ \vvv(0sr)
\end{center}

We print out all possible values \vvv(1sr) and \vvv(3sr) depending on \vvv(0sr) and \vvv(2sr).

\medskip

\begin{center}
	\begin{tabular}{|c|c|c|c|c|}
		\hline
		\vvv(0sr) & \vvv(2sr) & \vvv(1sr) & \vvv(3sr) \\
		\hline
		\hline
		2 & 2 & 2 & 2 \\
		\hline
		2 & 1 & 2,1 & 2,1 \\
		\hline
		2 & 0 & 2,1,0 & 2,1,0 \\
		\hline
		1 & 1 & 1 & 1 \\
		\hline
		1 & 0 & 1,0 & 1,0 \\
		\hline
		0 & 0 & 0 & 0 \\
		\hline
	\end{tabular}
	\medskip
	
	Table 4.1: Possible values 0sr, 1sr, 2sr, 3sr
\end{center}

All possible relationships between properties
\begin{center}
	2s, 2r, 4s, 4r, 4'r, 5s, 5r, 5'r
\end{center}
shows the following diagram.

$$\begin{array}{ccccccccccccccc}
&&\circled{5'r} && \circled{4'r}\\
&&\Uparrow && \Uparrow\\
\circled{2r}&\Rightarrow &\circled{5r}&\Rightarrow&\circled{4r} \\
\Downarrow && \Downarrow && \Downarrow \\
\circled{2s}&\Rightarrow &\circled{5s}&&\circled{4s}
\end{array}$$

Let's start with the relationship between 4sr and 4'r. Note that if \vvv(4r) = 0 then \vvv(4'r) can be any value, and if \vvv(4r) = 1 then \vvv(4'r) = 1. We present this relationship in the table:

\begin{center}
	\begin{tabular}{|c|c|c|}
		\hline
		\vvv(4sr) & \vvv(4'r) \\
		\hline
		\hline
		2 & 1 \\
		\hline
		1 & 1,0 \\
		\hline 
		0 & 1,0 \\
		\hline
	\end{tabular}
	\medskip
	
	Table 4.2: Possible values 4'r for the given value 4sr.
\end{center}

Let's move on to the relationship between 5sr, 5'r and 4sr. Note that if \vvv(5r) = 0 then \vvv(5'r) and \vvv(4r) can be of any value, and if \vvv(5r) = 1 then \vvv(5'r) = 1 and \vvv(4r) = 1. We present these relationships in the table:

\begin{center}
	\begin{tabular}{|c|c|c|c|c|}
		\hline
		\vvv(5sr) & \vvv(5'r) & \vvv(4sr)\\
		\hline
		\hline
		2 & 1 & 2\\
		\hline
		1 & 1,0 & 2,1,0\\
		\hline
		0 & 1,0 & 2,1,0\\
		\hline
	\end{tabular}
	
	\medskip
	
	Table 4.3: Possible values 5'r, 4sr for the given value 5sr.
\end{center}

Note that the implications between 2s, 2r, 5s, 5r are exactly when \vvv(2sr) $\leqslant$ \vvv(5sr) (Lemma \ref{444}), so we present these relationships in the table below:

\begin{center}
	\begin{tabular}{|c|c|c|}
		\hline
		\vvv(2sr) & \vvv(5sr) \\
		\hline
		\hline
		2 & 2 \\
		\hline
		1 & 2,1 \\
		\hline 
		0 & 2,1,0 \\
		\hline
	\end{tabular}
	
	\medskip
	
	Table 4.4: Possible values 5sr for the given value 2sr.
\end{center}

Tables 4.2, 4.3, 4.4 can be combined into one table. Let L1 denote the number of possible systems of values 4sr, 4'r, 5sr, 5'r for the given value of 2sr.

\begin{center}
	\begin{tabular}{|c|c|c|c|c||c|}
		\hline
		\vvv(2sr) & \vvv(5sr) & \vvv(5'r) & \vvv(4sr) & \vvv(4'r) & L1 \\
		\hline
		\hline
		2 & 2 & 1 & 2 & 1 & 1 \\
		\hline
		\multirow{3}{*}{1} & 2 & 1 & 2 & 1 & \multirow{3}{*}{1+2+8=11}\\
		\cline{2-5} & \multirow{2}{*}{1} & \multirow{2}{*}{1,0} & 2 & 1 &\\
		\cline{4-5} &&& 1,0 & 1,0 &\\
		\hline
		\multirow{3}{*}{0} & 2 & 1 & 2 & 1 & \multirow{3}{*}{1+4+16=21}\\
		\cline{2-5} & \multirow{2}{*}{1,0} & \multirow{2}{*}{1,0} & 2 & 1 &\\
		\cline{4-5} &&& 1,0 & 1,0 &\\
		\hline
	\end{tabular}
	
	\medskip
	
	Table 4.5: Possible values 5sr, 5'r, 4sr, 4'r for the given 2sr.\\
	
\end{center}

All possible relationships between properties
\begin{center}
	3s, 3r, 6s, 6r
\end{center}
shows the following diagram.

$$\begin{array}{ccccccccccccccc}
\circled{3r} & \Rightarrow & \circled{6r}\\
\Downarrow && \Downarrow\\
\circled{3s} & \Rightarrow & \circled{6s}
\end{array}$$

From the Lemma \ref{444} we know that the above implications apply exactly when \vvv(3sr) $\leqslant$ \vvv(6sr), so we present these dependencies in the table below.
Let L2 denote the number of possible systems of values 6sr for a given value of 3sr.

\begin{center} 
	\begin{tabular}{|c|c||c|}
		\hline
		\vvv(3sr) & \vvv(6sr) & L2 \\
		\hline
		\hline
		2 & 2 & 1\\
		\hline
		1 & 2,1 & 2\\
		\hline
		0 & 2,1,0 & 3\\
		\hline
	\end{tabular}
	
	\medskip
	
	Table 4.6: Possible values 6sr for the given value 3sr.
\end{center}

In Proposition \ref{p3.4.1} there are no 4's and 5's relationships. 4's and 5's can take any value.

\begin{center} 
	\begin{tabular}{|c|c|}
		\hline
		\vvv(4's) & \vvv(5's) \\
		\hline
		\hline
		1 & 1 \\
		\hline
		1 & 0 \\
		\hline
		0 & 1 \\
		\hline
		0 & 0 \\
		\hline
	\end{tabular}
	
	\medskip
	
	Table 4.7: Possible values 4's and 5's.
\end{center}

Based on the tables 4.1, 4.5, 4.6, 4.7 we can now determine the number of possible systems of values 1sr, 3sr, 4sr, 4's, 4'r, 5sr, 5's, 5'r, 6sr for a given system of values 0sr, 2sr. Let us denote this number by L3.

\medskip

\begin{center}
	\begin{tabular}{|c|c|c|c||c|}
		\hline
		\vvv(0sr) & \vvv(2sr) [L1] & \vvv(1sr) & \vvv(3sr) [L2] & L3 \\
		\hline
		\hline
		2 & 2 [1] & 2 & 2 [1] & 4 \\
		\hline
		2 & 1 [11] & 2,1 & 2 [1], 1 [2] & 264 \\
		\hline
		2 & 0 [21] & 2,1,0 & 2 [1], 1 [2], 0 [3] & 1512 \\
		\hline
		1 & 1 [11] & 1 & 1 [2] & 88 \\
		\hline
		1 & 0 [21] & 1,0 & 1 [2], 0 [3] & 840 \\
		\hline
		0 & 0 [21] & 0 & 0 [3] & 252 \\
		\hline
	\end{tabular}
	
	\medskip
	
	Table 4.8: Numbers of possible systems values 1sr, 3sr, 4sr, 4's, 4'r, 5sr, 5's, 5'r, 6sr for a given system of values 0sr, 2sr.
\end{center}

All the values in the L3 column were multiplied by $4$ because we included the 4's and 5's properties, which are independent of the other values (Table 4.7).
Summing up all the values of L3, we get $2960$ of possible sets of values.

\medskip

Recall that in an atomic monoid the condition $0s$ (Proposition \ref{p3.4.6}) is satisfied. Then \vvv (0sr) can be either $2$ or $1$. Therefore, from table 4.8, discard those monoids for which \vvv (0sr) $ = $ 0 (last row). These dependencies are presented in the table:

\begin{center}
	\begin{tabular}{|c|c|c|c||c|}
		\hline
		\vvv(0sr) & \vvv(2sr) & \vvv(1sr) & \vvv(3sr) & L3 \\
		\hline
		\hline
		2 & 2 & 2 & 2 & 4 \\
		\hline
		2 & 1 & 2,1 & 2,1 & 264 \\
		\hline
		2 & 0 & 2,1,0 & 2,1,0 & 1512 \\
		\hline
		1 & 1 & 1 & 1 & 88 \\
		\hline
		1 & 0 & 1,0 & 1,0 & 840 \\
		\hline
	\end{tabular}
	
	\medskip
	
	Table 4.9: Numbers of possible systems of values 1sr, 3sr, 4sr, 4's, 4'r, 5sr, 5's, 5'r, 6sr for a given system of values 0sr, 2sr in atomic monoids.
\end{center}

Thus, all possible systems of values in an atomic monoid, there is $2708$.

\medskip

Similarly, we consider all possible systems of values for ACCP-monoids, SR-monoids, pre-Schreier monoids, $\GCD$-monoids, $\GCD s$-monoids. The final tables are presented below.

\begin{center}
	\begin{tabular}{|c|c|c|c||c|}
		\hline
		\vvv(0sr) & \vvv(2sr) [L1] & \vvv(1sr) & \vvv(3sr) [L2] & L3 \\
		\hline
		\hline
		2 & 2 [2] & 2 & 2 [1] & 4 \\
		\hline
		2 & 1 [20] & 2,1 & 1 [1] & 80 \\
		\hline
		2 & 0 [20] & 2,1,0 & 1 [1] & 120 \\
		\hline
		1 & 1 [20] & 1 & 1 [1] & 40 \\
		\hline
		1 & 0 [20] & 1,0 & 1 [1] & 80 \\
		\hline
	\end{tabular}
	
	\medskip
	
	Tabela 4.10: Numbers of possible systems of values 1sr, 3sr, 4sr, 4's, 4'r, 5sr, 5's, 5'r, 6sr for a given values 0sr, 2sr in ACCP-monoids.
\end{center}

All the values in L3 were multiplied by $2$ because we included $4's$ which is independent of the other values.
Summing up all the values of L3, we get $324$ of possible ACCP-monoid systems.

\medskip 

In SR-monoids, pre-Schreier monoids, $\GCD$-monoids, $GCDs$-monoids the concept of a square-free element is equivalent to a radical generator. Therefore, instead of Asr, we will consider As for A = 1, 2, 3, 4, 4 ', 5, 5', 6.
We will determine all possible logical values of dependencies collected in the Proposition \ref{p3.4.1a}.

\medskip

\begin{center}
	\begin{tabular}{|c|c|c|c||c|}
		\hline
		\vvv(0s) & \vvv(2s) [L1] & \vvv(1s) & \vvv(3s) [L2] & L3 \\
		\hline
		\hline
		1 & 1 [1] & 1 & 1 [1] & 1 \\
		\hline
		1 & 0 [7] & 1,0 & 1 [1], 0 [2] & 42 \\
		\hline
		0 & 0 [7] & 0 & 0 [2] & 14 \\
		\hline
	\end{tabular}
	
	\medskip
	
	Table 4.11: Numbers of possible systems of values 1s, 3s, 4s, 4's, 5s, 5's, 6s for a given system of values 0s, 2s in SR-monoids.
\end{center}

\medskip

Summing up all the values of L3, we get $57$ of possible systems of values in a SR-monoids.

\medskip

\begin{center}
	\begin{tabular}{|c|c|c|c||c|}
		\hline
		\vvv(0s) & \vvv(2s) [L1] & \vvv(1s) & \vvv(3s) [L2] & L3 \\
		\hline
		\hline
		1 & 1 [1] & 1 & 1 [1] & 1 \\
		\hline
		1 & 0 [5] & 0 & 1 [1], 0 [2] & 15 \\
		\hline
		0 & 0 [5] & 0 & 0 [2] & 10 \\
		\hline
	\end{tabular}
	
	\medskip
	
	Table 4.12: Number of possible systems of values 0s, 3s, 4s, 4's, 5s, 5's, 6s for a given system of values 1s, 2s in pre-Schreier monoids.
\end{center}

\medskip

Summing up all the values of L3, we get $26$ of possible systems of values in a pre-Schreier monoids.

\medskip

\begin{center}
	\begin{tabular}{|c|c|c|c||c|}
		\hline
		\vvv(0s) & \vvv(2s) [L1] & \vvv(1s) & \vvv(3s) [L2] & L3 \\
		\hline
		\hline
		1 & 1 [1] & 1 & 1 [1] & 1 \\
		\hline
		0 & 0 [5] & 0 & 0 [2] & 10 \\
		\hline
	\end{tabular}
	
	\medskip
	
	Table 4.13: Number of possible systems of values 4s, 4's, 5s, 5's, 6s for a given system of values 0s, 1s, 2s, 3s in a $\GCD$-monoids.
\end{center}

\medskip

Summing up all the values of L3, we get $11$ of possible systems of values in a $\GCD$-monoids.

\medskip

We know that any GCDs-monoid satisfies the $5's$ condition (Proposition \ref{p3.4.5}). Therefore, the same relationships apply as for the~$\GCD$-monoids, but \vvv(5's) $ = $ $1$ should be included.

\medskip

\begin{center}
	\begin{tabular}{|c|c|c|c||c|}
		\hline
		\vvv(0s) & \vvv(2s) [L1] & \vvv(1s) & \vvv(3s) [L2] & L3 \\
		\hline
		\hline
		1 & 1 [1] & 1 & 1 [1] & 1 \\
		\hline
		0 & 0 [3] & 0 & 0 [2] & 6 \\
		\hline
	\end{tabular}
	
	\medskip
	
	Table 4.14: Numbers of possible systems of values 4s, 4's, 5s, 5's, 6s for a given system of values 0s, 1s, 2s, 3s.
\end{center}

\medskip

Summing up all the values of L3, we get $7$ of possible systems of values in a GCDs-monoids.

\end{document}